\documentclass{amsart}
\usepackage{amsfonts}
\usepackage{amscd}
\usepackage{amssymb}
\usepackage{mathrsfs}
\usepackage[all]{xy}
\theoremstyle{plain}
\newtheorem{thm}{Theorem}[section]
\newtheorem{cor}[thm]{Corollary}

\newtheorem{lem}[thm]{Lemma}
\newtheorem{pro}[thm]{Proposition}
\theoremstyle{definition}
\newtheorem{ex}[thm]{Example}

\newtheorem{defn}[thm]{Definition}
\usepackage{amsmath}
\usepackage{amssymb}
\usepackage{amsthm}
\usepackage{latexsym}
\usepackage[T1]{fontenc}
\usepackage[all]{xy}

\DeclareMathOperator{\Per}{Per} \DeclareMathOperator{\supp}{supp}

\DeclareMathOperator{\id}{id}
\DeclareMathOperator{\Ad}{Ad}
\DeclareMathOperator{\diag}{diag}
\DeclareMathOperator{\PIP}{PIP}

\author{Christian Svensson}
\address{Mathematical Institute, Leiden University,
P.O. Box 9512, 2300 RA Leiden, The Netherlands, and Centre for
Mathematical Sciences, Lund University, Box 118, SE-221 00 Lund,
Sweden} \email{chriss@math.leidenuniv.nl}
\author{Jun Tomiyama}
\address{Department of Mathematics, Tokyo Metropolitan University, Minami-Osawa, Hachioji City, Japan} \email{juntomi@med.email.ne.jp}
\title[On the commutant of $C(X)$ in $C^*$-crossed products by $\mathbb{Z}$]{On the commutant of $C(X)$ in $C^*$-crossed  products by $\mathbb{Z}$ and their representations}
\begin{document}
\begin{abstract}
For the $C^*$-crossed product $C^*(\Sigma)$ associated with an arbitrary topological dynamical system $\Sigma = (X, \sigma)$, we provide a detailed analysis of the commutant, in $C^* (\Sigma)$, of $C(X)$ and the commutant of the image of $C(X)$ under an arbitrary Hilbert space representation $\tilde{\pi}$ of $C^* (\Sigma)$. In particular, we give a concrete description of these commutants, and also determine their spectra. We show that, regardless of the system $\Sigma$, the commutant of $C(X)$ has non-zero intersection with every non-zero, not necessarily closed or self-adjoint, ideal of $C^* (\Sigma)$. We also show that the corresponding statement holds true for the commutant of $\tilde{\pi}(C(X))$ under the assumption that a certain family of pure states of $\tilde{\pi}(C^* (\Sigma))$ is total. Furthermore we establish that, if $C(X) \subsetneq C(X)'$, there exist both a $C^*$-subalgebra properly between $C(X)$ and $C(X)'$ which has the aforementioned intersection property, and such a $C^*$-subalgebra which does not have this property. We also discuss existence of a projection of norm one from $C^*(\Sigma)$ onto the commutant of $C(X)$.
\end{abstract}
\maketitle
\section{Introduction}
Let $\Sigma = (X, \sigma)$ be a topological dynamical system where $X$ is a compact Hausdorff space and $\sigma$ is a homeomorphism 
of $X$. We denote by $\alpha$ the automorphism of $C(X)$, the algebra of all continuous complex valued functions on $X$, induced by 
$\sigma$, namely $\alpha(f) = f \circ \sigma^{-1}$ for $f \in C(X)$. Denote by $C^* (\Sigma)$ the associated transformation group 
$C^*$-algebra, that is, the $C^*$-crossed product of $C(X)$ by $\mathbb{Z}$, where $\mathbb{Z}$ acts on $C(X)$ via iterations of $\alpha$. The interplay between topological dynamical systems and $C^*$-algebras has been intensively studied, see for example~\cite{STom4, T24, T34, T44, T64}. The following result constitutes the motivating background of this paper.

\begin{thm}\cite[Theorem 5.4]{T24}\label{triquivc4}
For a topological dynamical system $\Sigma$, the following statements are equivalent.
\begin{enumerate}
 \item $\Sigma$ is topologically free;
\item $I\cap C(X) \neq 0$ for every non-zero closed ideal $I$ of $C^*(\Sigma)$;
\item $C(X)$ is a maximal abelian $C^*$-subalgebra of $C^*(\Sigma)$.
\end{enumerate}
\end{thm}
Recall that a system $\Sigma = (X, \sigma)$ is called topologically free if the set of its non-periodic points is dense in $X$.
We say that $C(X)$ has the intersection property for closed ideals of $C^*(\Sigma)$ when (ii) is satisfied (cf. Definition~\ref{defintpro4}).

Many significant results concerning the interplay between $\Sigma$ and $C^*(\Sigma)$ have been obtained under the assumption that $\Sigma$ is topologically free.
As there are important examples of topological dynamical systems that are not topologically free, rational rotation of the unit circle being a typical one, our aim is to analyze the situation around Theorem~\ref{triquivc4} for arbitrary $\Sigma$.

We shall be concerned in detail with the commutant of $C(X)$, which we denote by $C(X)'$, and the commutant of the image of $C(X)$ 
under Hilbert space representations of $C^*(\Sigma)$. In a series of papers, \cite{SSD14, SSD24, SSD34}, improved analogues of 
Theorem~\ref{triquivc4} have been obtained in the context of an algebraic crossed product by the integers of, in particular, commutative Banach algebras $A$ more general than $C(X)$ by the integers, and especially $A'$ has been thoroughly investigated there. In that setup it is an elementary result that $A'$ is commutative (\cite[Proposition 2.1]{SSD14}) and thus a maximal commutative subalgebra of the corresponding crossed product, and that $A'$ has non-zero intersection with every non-zero ideal (\cite[Theorem 3.1]{SSD34}) even when $A$ is an arbitrary commutative complex algebra.
 
Here we give an explicit description of $C(X)' \subseteq C^*(\Sigma)$ and $\pi(C(X))' \subseteq \tilde{\pi}(C^*(\Sigma))$ where $\tilde{\pi} = \pi \times u$ is a Hilbert space representation of $C^*(\Sigma)$ (Proposition~\ref{commbesk4}).
Moreover, we prove that these algebras constitute commutative, hence maximal commutative, $C^*$-subalgebras that are invariant under $\Ad \delta$ (recall that, for $a \in C^* (\Sigma)$, $\Ad \delta (a) = \delta a \delta^{*}$)  and $\Ad \delta^*$ respectively under $\Ad u$ and $\Ad u^*$ (here $\delta$ denotes the canonical unitary element of $C^*(\Sigma)$ that implements the action of $\mathbb{Z}$ on $C(X)$ via $\alpha$, and $u = \tilde{\pi} (\delta)$), and determine the structure of their spectra (Theorem~\ref{comcom4}). Invariance under these automorphisms implies that their restrictions to $C(X)'$ and $\pi(C(X))'$, respectively, correspond to homeomorphisms of the spectra of these algebras. Certain aspects of the associated dynamical systems are investigated (Proposition \ref{unutv4}) and later used to prove Theorem~\ref{intpro4}: $\pi(C(X))'$ has the intersection property for ideals, not necessarily closed or self-adjoint, under the assumption that a certain family of pure states of $\tilde{\pi}(C^*(\Sigma))$ is total and, consequently, Corollary~\ref{intprocol4}, one of the main results of this paper: regardless of the system $\Sigma$, $C(X)'$ has the intersection property for ideals of $C^* (\Sigma)$. It is a consequence of Proposition~\ref{sammaskar4} that these algebras have the intersection property for arbitrary ideals rather than just for closed ones, which also sharpens our background result, Theorem~\ref{triquivc4}. In Section~\ref{intsub4} we investigate ideal intersection properties of so called intermediate subalgebras, meaning $C^*$-subalgebras $B$ of $C^*(\Sigma)$ such that $C(X) \subseteq B \subseteq C(X)'$. In Proposition~\ref{specint4} we give an abstract condition on such $B$, in terms of the relation of its spectrum to the spectrum of $C(X)'$, that it satisfies precisely when it has the intersection property for ideals. Using this, we provide in Theorem~\ref{triquivcny4} a rather short alternative proof of a refined version of our aforementioned background result. In Theorem~\ref{mellankompl4} we clarify completely the situation concerning $C(X)$ and its commutant, as well as the intersection property for ideals of these algebras and intermediate subalgebras: $C(X)'$ is the unique maximal abelian $C^*$-subalgebra of $C^* (\Sigma)$ containing $C(X)$, $C(X)' = C(X)$ precisely when $\Sigma$ is topologically free and as soon as $C(X)'$ is strictly larger than $C(X)$ there exist intermediate subalgebras $B_i$ with $C(X) \subsetneq B_i \subsetneq C(X)'$, for $i=1,2$, such that $B_1$ has the intersection property for ideals and $B_2$ does not. Finally, in Section~\ref{projn4} we discuss the existence of norm one projections from $C^*(\Sigma)$ onto $C(X)'$.
\section{Notation and preliminaries}\label{notprel4}
Throughout this paper we consider topological dynamical systems $\Sigma = (X, \sigma)$ where $X$ is a compact Hausdorff space and $\sigma : X \rightarrow X$ is a homeomorphism. Here $\mathbb{Z}$ acts on $X$ via iterations of $\sigma$, namely $x \stackrel{n}{\mapsto} \sigma^n (x)$ for $n \in \mathbb{Z}$ and $x \in X$. 
We denote by $\Per^{\infty}(\sigma)$ and
 $\Per(\sigma)$ the sets of aperiodic points and periodic points,
 respectively. If $\Per^{\infty} (\sigma) = X$, $\Sigma$ is called free and if $\Per^{\infty} (\sigma)$ is dense in $X$, $\Sigma$ is called topologically free.  Moreover,
for an integer $n$ we write 
$\Per^n(\sigma) = \Per^{-n}(\sigma) = \{ x \in X : \sigma^n(x) =
 x\}\), and $\Per_n(\sigma)$ for the set of all points belonging to $\Per^n (\sigma)$ but to no $\Per^k (\sigma)$ with $\vert k \vert$ non-zero and strictly less than $\vert n \vert$. When $n = 0$ we regard $\Per^0(\sigma) = X$.
We write $\Per(x) = k$ if $x \in \Per_k (\sigma)$. Note that if $\Per(y) = k$ and $y \in \Per^n (\sigma)$ then $k | n$.
For a subset $S \subseteq X$ we write its interior as $S^0$ and its closure as $\bar{S}$.
When a periodic point $y$ belongs to the interior of
 $\Per_k(\sigma)$ for some integer $k$ we call $y$ a periodic interior point. We denote the set
 of all such 
points by $\PIP(\sigma)$. Note that $\PIP(\sigma)$ does not coincide with $\Per(\sigma)^0$ in general, as the following example shows.
Let $X = [0,1]\times [-1,1]$ be endowed with the standard subspace topology from $\mathbb{R}^2$ and let $\sigma$ be the homeomorphism of $X$ defined as reflection in the $x$-axis. Then clearly $X = \Per(\sigma) = \Per(\sigma)^0 = \Per^2 (\sigma)$. Furthermore, $\Per_1 (\sigma) = [0,1] \times \{0\}$ which is a closed subset of $X$ with empty interior so $\PIP(\sigma) = \Per_2(\sigma)$. Hence $\Per_1 (\sigma) = \Per(\sigma) \setminus \PIP(\sigma)$. As we just saw that $\Per_2 (\sigma)$ is open and non-empty, this  example also shows that the sets $\Per_n (\sigma)$ are in general not closed, as opposed to the sets $\Per^n (\sigma)$ which are easily seen to always be closed. The following topological lemma will be a key result in what follows.
\begin{lem}\label{untat4} The union of $\Per^{\infty}(\sigma)$ and $\PIP(\sigma)$ is dense
 in $X$.
\end{lem}
\begin{proof} Suppose the union were not dense in $X$, and let $Y$ be
 the complement of its closure. Then $Y$ is a non-empty open subset of $X$ and hence it is locally compact in the induced topology.
 Since
a locally compact space is a Baire space and $Y = \cup_{k=1}^{\infty} \Per^k (\sigma) \cap Y$, where the $\Per^k (\sigma) \cap Y$ are closed in $Y$, there exists a positive integer $n$ such that $\Per^n(\sigma) \cap Y$ has non-empty interior in $Y$ and hence in $X$ as $Y$ is an open subset of $X$. Take the minimal such integer and write
 it as $n$ again. If $n=1$ we arrive at a contradiction as $\Per^1 (\sigma) = \Per_1 (\sigma)$ and the above then implies that $\Per_1 (\sigma)$ has non-empty interior in $X$. Thus we assume that $n >1$. Let $k_1,k_2,\ldots,k_i$ be the  positive divisors of $n$ that are strictly smaller than $n$. 
Suppose that $\cup_{j=1}^i \Per^{k_j}(\sigma)\cap Y$ has non-empty interior, say $\emptyset \neq V \cap Y \subseteq \cup_{j=1}^i \Per^{k_j}(\sigma)\cap Y$ for some open subset $V$ of $X$, so that $V \cap Y = \cup_{j=1}^i \Per^{k_j}(\sigma)\cap V \cap Y$. Since $V \cap Y$, being open in $X$, is a locally compact Hausdorff space in the induced topology, hence a Baire space, there exists $j$ such that the closure (in $V \cap Y$) of $\Per^{k_j} (\sigma) \cap V \cap Y$ has non-empty interior in $V \cap Y$. However, since $\Per^{k_j} (\sigma) \cap V \cap Y$ is closed in $V \cap Y$, because $\Per^{k_j} (\sigma)$ is closed in $X$, and since $V \cap Y$ is open in $X$, we see that $\Per^{k_j} (\sigma) \cap V \cap Y$ itself has an interior point in $X$. Hence $\Per^{k_j} (\sigma)$ has an interior point in $Y$ and using the assumption on $n$ we arrive at a contradiction since $k_j < n$.
We conclude that $\cup_{j=1}^i \Per^{k_j}(\sigma)\cap Y$ has empty interior.
Denote by $U$ the interior of $\Per^n (\sigma) \cap Y$ in $Y$. Since $U$ is non-empty by assumption, it follows from the above that $U \setminus (\cup_{j=1}^i \Per^{k_j}(\sigma) \cap Y)$ is a non-empty open subset of $Y$ and hence of $X$.
But $U \setminus (\cup_{j=1}^i \Per^{k_j}(\sigma) \cap Y) \subseteq \Per^n(\sigma) \cap Y \setminus (\cup_{j=1}^i \Per^{k_j}(\sigma) \cap Y) = \Per_n (\sigma) \cap Y$.
We conclude that $\Per_n(\sigma)$ has an interior point in $Y$, which is a contradiction.
\end{proof}

We remark that when speaking of ideals we shall always mean two-sided ideals which are not necessarily closed or self-adjoint unless we state this explicitly.

With the automorphism $\alpha : C(X) \rightarrow C(X)$ defined by $\alpha(f) = f \circ \sigma^{-1}$ for $f \in C(X)$, we denote the $C^*$-crossed product $C(X) \rtimes_{\alpha} \mathbb{Z}$ by $C^* (\Sigma)$. For simplicity, we denote the natural isomorphic copy of $C(X)$ in $C^*(\Sigma)$ by $C(X)$ as well. We denote the canonical unitary element of $C^* (\Sigma)$ that implements the action of $\mathbb{Z}$ on $C(X)$ via $\alpha$ by $\delta$, recalling that $\alpha(f) = \Ad \delta (f)$ for $f \in C(X)$. By construction, $C^*(\Sigma)$ is generated as a $C^*$-algebra by $C(X)$ together with $\delta$. A generalized polynomial is a finite sum of the form $\sum_{n} f_n \delta^n$ with $f_n \in C(X)$, and we shall refer to the norm-dense $*$-subalgebra of $C^* (\Sigma)$ consisting of all generalized polynomials as the algebraic part of $C^* (\Sigma)$. 
We write the canonical faithful projection of norm one from $C^* (\Sigma)$ to $C(X)$ as $E$ and recall that $E$ is defined on the algebraic part of $C^* (\Sigma)$ as $E(\sum_n f_n \delta^n) = f_0$.
For an element $a \in C^*(\Sigma)$ and an integer $j$  we define $a(j) = E(a \delta^{-j})$, the $j$-th generalized Fourier coefficient of $a$. It is a fact that $a = 0$ if and only if $a(j) = 0$ for all integers $j$ and that $a$ is thus uniquely determined by its generalized Fourier coefficients (\cite[Theorem 1.3]{T24}). A Hilbert space representation of $C^* (\Sigma)$ is written as $\tilde{\pi} = \pi \times u$, where $\pi$ is the representation of $C(X)$ on the same Hilbert space given by restriction of $\tilde{\pi}$, and $u = \tilde{\pi}(\delta)$. The operations on $C^* (\Sigma)$ then imply that
\[\pi(\alpha(f)) = u \pi(f)u^* = \Ad u (\pi(f)), \textup{ for } f \in C(X).\]
We shall often make use
 of the
 dynamical system \(\Sigma_{\pi} = (X_{\pi},\sigma_{\pi})\) derived
 from a representation $\tilde{\pi}$. As explained in \cite[p.26]{T24}, we
 define this dynamical 
system as
\[
 X_{\pi} = h(\pi^{-1}(0)) \quad \mbox{and\quad\( \sigma_{\pi} = 
\sigma_{\upharpoonright X_{\pi}}\)},
\]
where $h(\pi^{-1}(0))$ means the standard hull of the kernel ideal of
 $\pi$ in $C(X)$; $X_{\pi}$ is obviously a closed subset of $X$ that is invariant under $\sigma$ and its inverse. The system
 $\Sigma_{\pi}$ is topologically conjugate to the dynamical system
 \(\Sigma'_{\pi} = (X'_{\pi},\sigma_{\pi}')\) where $X_{\pi}'$ is the
 spectrum of 
$\pi(C(X))$ and the map $\sigma_{\pi}'$ is the homeomorphism of
 $X_{\pi}'$ induced by the automorphism $\Ad u$ on $\pi(C(X))$. To see this, note that the homeomorphism $\theta : X_{\pi} \rightarrow X_{\pi}'$ induced by the isomorphism $\pi(f) \mapsto f_{\upharpoonright X_{\pi}}$ between $\pi(C(X))$ and $C(X_{\pi})$ is such that $\sigma_{\pi}' \circ \theta = \theta \circ \sigma_{\pi}$.  Thus we
 identify
 these two dynamical systems. Note that under this identification, $\pi(f)$ corresponds to the restriction 
of the function $f$ to the set $X_{\pi}$. We denote the canonical unitary element of $C^* (\Sigma_{\pi})$ by $\delta_{\pi}$. We now recall three results that will be important to us throughout this paper.
\begin{pro}\label{irrfri4}\cite[Proposition 3.4.]{T64}
If $\tilde{\pi} = \pi \times u$ is an infinite-dimensional irreducible representation of $C^* (\Sigma)$, then the dynamical system $\Sigma_{\pi}$ is topologically free.
\end{pro}
\begin{thm}\label{topfriprojn4}\cite[Theorem 5.1.]{T24}
Let $\tilde{\pi} = \pi \times u$ be a representation of $C^* (\Sigma)$ on a Hilbert space $H$. If the induced dynamical system $\Sigma_{\pi}$ is topologically free, then there exists a projection $\epsilon_{\pi}$ of norm one  from the $C^*$-algebra $\tilde{\pi} (C^* (\Sigma))$ to $\pi(C(X))$ such that the following diagram commutes.
\begin{displaymath}
\xymatrix{
C^*(\Sigma) \ar[r]^{\tilde{\pi}} \ar[d]_{E} &
\tilde{\pi}(C^*(\Sigma)) \ar[d]^{\epsilon_{\pi}}\\
C(X) \ar[r]_{\pi}  & \pi(C(X))}
\end{displaymath}
\end{thm}
\begin{cor}\label{bildkp4}\cite[Corollary 5.1.A.]{T24}
Suppose the situation is as in Theorem~\ref{topfriprojn4}. Then the map defined by $\pi(f) \mapsto f_{\upharpoonright X_{\pi}}$ for $f \in C(X)$ and $u \mapsto \delta_{\pi}$, extends to an isomorphism between $\tilde{\pi} (C^* (\Sigma))$ and $C^* (\Sigma_{\pi})$.
\end{cor}

For $x \in X$ we denote by $\mu_x$ the functional on $C(X)$ that acts as point evaluation in $x$. Since the pure state extensions to $C^* (\Sigma)$ of the point evaluations on $C(X)$ will play a prominent role in this paper, we shall now recall some basic facts about them, without proofs. For further details and proofs, we refer to~\cite[\S 4]{T24}. For $x \in \Per^{\infty} (\sigma)$ there is a unique pure state extension of $\mu_x$, denoted by $\varphi_x$, given by $\varphi_x = \mu_x \circ E$. The set of pure state extensions of $\mu_y$ for $y \in \Per(\sigma)$ is parametrized by the unit circle $\{\varphi_{y,t} : t \in \mathbb{T}\}$. We write the GNS-representations associated with the pure state extensions above as $\tilde{\pi}_x$ and $\tilde{\pi}_{y,t}$. For $x \in \Per^{\infty}(\sigma)$, $\tilde{\pi}_x$ is the representation of $C^* (\Sigma)$ on $\ell^2$, whose standard basis we denote by $\{e_i\}_{i \in \mathbb{Z}}$, defined on the generators as follows. For $f \in C(X)$ and $i \in \mathbb{Z}$ we have $\tilde{\pi}_{x} (f) e_i = f \circ \sigma^i (x) \cdot e_i$, and $\tilde{\pi}_{x} (\delta) e_i = e_{i+1}$. For $y \in \Per(\sigma)$ with $\Per(y) = p$ and $t \in \mathbb{T}$, $\tilde{\pi}_{y,t}$ is the representation on $\mathbb{C}^p$, whose standard basis we denote by $\{e_i\}_{i=0}^{p-1}$, defined as follows. For $f \in C(X)$ and $i \in \{0,1,\ldots, p-1\}$ we set $\tilde{\pi}_{y,t} (f) e_i = f \circ \sigma^i (y) \cdot e_i$. For $j \in \{0, 1, \ldots, p-2\}$, $\tilde{\pi}_{y,t} (\delta) e_j = e_{j+1}$ and $\tilde{\pi}_{y,t}(\delta) e_{p-1} = t \cdot e_0$. We also mention that the unitary equivalence class of $\tilde{\pi}_x$ is determined by the orbit of $x$, and that of $\tilde{\pi}_{y,t}$ by the orbit of $y$ and the parameter $t$.

In what follows, we shall sometimes make use of the following important result (\cite[Proposition 2]{T44}).
\begin{pro}\label{idbesk4}
For every closed ideal $I \subseteq C^* (\Sigma)$ there exist families $\{x_{\alpha}\}$ of aperiodic points and $\{y_{\beta}, t_{\gamma}\}$ of periodic points and parameters from the unit circle, such that $I$ is the intersection of the associated kernels $\ker(\tilde{\pi}_{x_{\alpha}})$ and $\ker(\tilde{\pi}_{y_{\beta}, t_{\gamma}})$.
\end{pro}
\section{The structure of $C(X)'$ and $\pi(C(X))'$}
We shall now make a detailed analysis of the commutants $C(X)'$ of $C(X) \subseteq C^* (\Sigma)$ and $\pi(C(X))'$ of  $\pi(C(X)) \subseteq \tilde{\pi} (C^* (\Sigma))$, respectively. Here $\tilde{\pi} = \pi \times u$ is a Hilbert space representation of $C^* (\Sigma)$ as usual.
These $C^*$-subalgebras are defined as follows
\[C(X)' = \{a \in C^* (\Sigma) : af = fa \textup{ for all } f \in C(X)\}\] and \[\pi(C(X))' = \{\tilde{\pi} (a) \in \tilde{\pi} (C^* (\Sigma)) : \tilde{\pi} (af) = \tilde{\pi} (fa) \textup{ for all } f \in C(X)\}.\]
We shall need the following topological lemma.
\begin{lem}\label{bairegrej4}
The system $\Sigma = (X, \sigma)$ is topologically free if and only if $\Per^{n} (\sigma)$ has empty interior
for all positive integers $n$.
\end{lem}
\begin{proof}
Clearly, if there is a positive integer $n_0$ such that $\Per^{n_0}
(\sigma)$ has non-empty interior, the non-periodic points are not
dense. For the converse, we recall that $X$ is a Baire space
since it is compact and Hausdorff, and note that
\[\Per (\sigma) = \bigcup_{n > 0} \Per^n (\sigma).\]
If $\Per^{\infty} (\sigma)$ is not dense, its complement $\Per (\sigma)$ has 
non-empty interior, and as the sets $\Per^n (\sigma)$ are clearly
all closed, there must exist an integer $n_0 > 0$ such that
$\Per^{n_0} (\sigma)$ has non-empty interior since $X$ is
a Baire space. \end{proof}
The following proposition describes the commutants $C(X)'$ and $\pi(C(X))'$.
\begin {pro}\label{commbesk4}
Let $\tilde{\pi} = \pi \times u$ be a Hilbert space representation of $C^* (\Sigma)$.
%Letting $\alpha, \beta, \gamma$ be the parameters appearing in the ideal $I = \ker(\tilde{\pi})$ in accordance with Theorem~\ref{idbesk}, we have
\begin{enumerate}
\item $C(X)' = \{a \in C^*(\Sigma) : \supp (a(n))\subseteq \Per^n(\sigma)
 \textup{ for all }n\}$.
Consequently, $C(X)'= C(X)$ if and only if the dynamical system is
 topologically free.

\item $\pi(C(X))'$ consists of all elements $\tilde{\pi} (a)$ such that $\tilde{\pi}_{x_{\alpha}} (a) \in \pi_{x_{\alpha}}(C(X))$ and
 $\tilde{\pi}_{y_{\beta}, t_{\gamma}} (a) \in \pi_{y_{\beta},t_{\gamma}}(C(X))$ 
for all $\alpha,\beta,\gamma$ that appear in the description of the ideal $I=\ker(\tilde{\pi})$ as in Proposition~\ref{idbesk4}.
\end{enumerate}
\end{pro}
\begin{proof}The first assertion is a direct extension of~\cite[Corollary 3.4]{SSD14} to the context of $C^*$-crossed products. The main steps of the proof are contained in the first part of the
 proof of \cite [Theorem 5.4]{T24}, but we reproduce them here for the 
reader's convenience, and because this is also  our basic starting point. Let $a \in C^* (\Sigma)$ and $f \in C(X)$ be arbitrary. We then have, for $n \in \mathbb{Z}$,
\begin{align*}
(fa)(n) & = E(fa\delta^{\ast n}) = f \cdot E(a \delta^{*n}) = f \cdot a(n),\\
(af)(n) & = E(af \delta^{*n}) = E(a \delta^{*n} \alpha^n (f)) = E(a \delta^{*n}) \cdot \alpha^n (f) = a(n) \cdot f\circ \sigma^{-n}.
\end{align*}
Hence for $a \in C(X)'$ we have, for any $f \in C(X), x \in X$ and $n \in \mathbb{Z}$, that
\[f(x)\cdot a(n)(x) =  a(n)(x) \cdot f \circ \sigma^{-n} (x).\]
Therefore, if $a(n)(x)$ is not zero we have that $f(x) =
 f \circ \sigma^{-n} (x)$ for all $f \in C(X)$. It follows that $\sigma^{-n} (x) = x$ and
hence that $x$ belongs to the set $\Per^n(\sigma)$, whence $\supp(a(n))\subseteq
 \Per^n(\sigma)$.

Conversely, if $\supp(a(n)) \subseteq \Per^n (\sigma)$ for every $n$, it follows easily from the above that  \[(fa)(n) = (af)(n)\] 
for every $f \in C(X)$ and $n \in \mathbb{Z}$ and hence that $a$ belongs to $C(X)'$.

Moreover, by Lemma~\ref{bairegrej4}, $\Sigma$ is topologically
 free if and only if for every nonzero integer $n$ the set $\Per^n(\sigma)$
has empty interior. So when the system is topologically free, we see from the above description of $C(X)$ that an element $a$ in $C(X)'$  necessarily belongs to $C(X)$. If $\Sigma$ is not topologically free, however, some $\Per^n (\sigma)$ has non-empty interior and hence there is a non-zero function $f \in C(X)$ such that $\supp(f) \subseteq \Per^n (\sigma)$. Then $f \delta^n \in C(X)' \setminus C(X)$ by the above.

For the second assertion, note that for an element $a$ in $C^*(\Sigma)$, 
$\tilde{\pi}(a)$ belongs to $\pi(C(X))'$ if and only if $af - fa$ belongs
 to the
kernel 
$I$ for every function $f \in C(X)$. Hence this is equivalent to saying that the
 image of $a$ belongs to the commutant of the image of $C(X)$ for all 
irreducible representations with respect to the
indices $\alpha,\beta,\gamma$. This in turn is equivalent to the
 assertion in (ii) because when $x$ is aperiodic we know that $\tilde{\pi}_x$ is infinite-dimensional, whence it follows from Proposition~\ref{irrfri4} that $\Sigma_{\pi_x}$ is topologically free, and since $C(X_{\pi_x})$ corresponds to $\pi_x (C(X))$ under the isomorphism in Corollary~\ref{bildkp4}, part (i) of this proposition implies that $\pi_x (C(X))' = \pi_x (C(X))$. So $\tilde{\pi}_{x_{\alpha}} (a) \in \pi_{x_{\alpha}} (C(X))$ for all $\alpha$ that occur in the description of $I$. When $y$ is periodic with period $n$, the image $\pi_{y,t}(C(X))$ consists of the diagonal
matrices in $M_n$ and thus coincides with its commutant. We conclude that $\tilde{\pi}_{y_{\beta}, t_{\gamma}} (a) \in \pi_{y_{\beta}, t_{\gamma}} (C(X))$ for all $\beta, \gamma$ that occur in the description of $I$.
\end{proof}
Before continuing, we recall the following noncommutative version of Fejér's theorem on Ces\`{a}ro sums, which we shall use in our 
arguments. 
\begin{pro}\label{cesar4}\cite[Proposition 1]{T44}.
The sequence $\{\sigma_n (a)\}_{n =0}^{\infty}$, where $\sigma_n (a)$ is the $n$-th generalized Ces\`{a}ro sum of an element $a \in 
C^*(\Sigma)$, defined by
\[
\sigma_n(a) = \sum^n_{-n}( 1 - \frac{\mid i\mid}{n + 1})a(i)\delta^i, 
\]
converges to $a$ in norm.
\end{pro}
Actually it is known that replacing the Ces\`{a}ro sums by any other summability kernel such as the de la Vallée-Poussin  kernel, 
Jackson kernel etc, we obtain the corresponding approximation sequences 
converging to $a$ in norm (\cite[Proposition 1]{T44}).

In the passage following Corollary~\ref{bildkp4} we described the pure state extensions to $C^* (\Sigma)$ of the point evaluations on $C(X)$. Recalling the notation introduced there, we define the following two sets.
\[
\Phi= \{\varphi_x : x \in \Per^{\infty}(\sigma)\}\cup \{\varphi_{y,t} :
 y \in \Per(\sigma), t \in \mathbb{T}\}
.\]
\[\Phi' = \{\varphi_x : x \in \Per^{\infty}(\sigma)\}\cup \{\varphi_{y,t} : y \in \PIP(\sigma), t \in \mathbb{T}\}.
\]
We notice that a representation $\tilde{\pi}$ of $C^*(\Sigma)$ can be factored as 
$\tilde{\pi} = \hat{\pi}\circ \tilde{\rho}$, where $\tilde{\rho}$ is the canonical homomorphism
from $C^*(\Sigma)$ to $C^*(\Sigma_{\pi})$ induced by the restriction map from $C(X)$ to $C(X_{\pi})$ and $\hat{\pi}$ is the associated homomorphism from $C^*(\Sigma_{\pi})$ onto
 $\tilde{\pi}(C^*(\Sigma))$. Writing this out, we have
 \begin{align*}
 f \stackrel{\tilde{\rho}}{\mapsto} & f_{\upharpoonright X_{\pi}} \stackrel{\hat{\pi}}{\mapsto} \pi (f) \textup{ for } f \in C(X),\\
 \delta \stackrel{\tilde{\rho}}{\mapsto} & \delta_{\pi} \stackrel{\hat{\pi}}{\mapsto} \tilde{\pi} (\delta) = u.
 \end{align*}
Note that here the restriction of $\hat{\pi}$ to $C(X_{\pi})$ is an 
isomorphism onto $\pi(C(X))$.
In the following arguments we often use this factorization, and may then regard $C(X_{\pi})$ as an  embedded
 subalgebra of $\tilde{\pi}(C^*(\Sigma))= \hat{\pi}(C^*(\Sigma_{\pi}))$. Moreover, we
 consider the pure state extensions to $C^*(\Sigma_{\pi})$
 of point evaluations on $C(X_{\pi})$ as well as the pure state
 extensions to $\tilde{\pi}(C^*(\Sigma))$ of point evaluations on $C(X_{\pi})$ when the latter is viewed as an embedded
subalgebra of the former. 
We denote the families of pure state
 extensions to $C^*(\Sigma_{\pi})$ corresponding to $\Phi$ (the pure state extensions of all point evaluations on $C(X_{\pi})$) and to $\Phi'$ (the pure state extensions of point evaluations on $C(X_{\pi})$ in the set of points $\Per^{\infty}(\sigma_{\pi}) \cup \PIP(\sigma_{\pi})$) by $\Phi_{\pi}$ and $\Phi'_{\pi}$, respectively.
The family of pure state extensions to $\tilde{\pi}(C^*(\Sigma))$ of point evaluations on the aforementioned embedded copy of $C(X_{\pi})$ will be denoted by $\Phi(\tilde{\pi})$.

For the following arguments, we recall the notion of right multiplicative domain
 for a unital positive linear map $\tau$ between two unital
 $C^*$-algebras
$A$ and $B$. We write
%\[
%A^l_{\tau}= \{a \in A : \tau(xa) = \tau(x)\tau(a) \ \forall x\in A\},
%\]
\[
A^r_{\tau}= \{a \in A : \tau(ax) = \tau(a)\tau(x) \ \forall x\in A\},
\]
and call this set the right multiplicative domain of $\tau$.
For a detailed account of the theory of positive linear maps between $C^*$-algebras, we refer to~\cite{paulsen4}.
Using the fact that positive linear maps between unital $C^*$-algebras respect involution, the following right-sided version of~\cite[Theorem 3.1]{C4}, which concerns left multiplicative domains, is readily concluded.
\begin{thm}~\label{choi4}
If $\tau : A \rightarrow B$ is a $2$-positive linear map between two unital $C^*$-algebras $A$ and $B$, then $A^r_{\tau} = \{a \in A : \tau(aa^*) = \tau(a)\tau(a)^*\}$.
\end{thm}
Since a state of a unital $C^*$-algebra is even completely positive, the above result holds in particular when $\tau$ is a state. The Cauchy-Schwarz inequality for states on $C^*$-algebras implies that if a state vanishes on a positive element, $a$, then it also vanishes on its positive square root, $\sqrt{a}$. It follows from Theorem~\ref{choi4} that $\sqrt{a}$, and hence $a$, is in the right multiplicative domain of every state that vanishes on $a$.
We shall make use of right multiplicative domains in the following lemma. Recall that a family of states of a $C^*$-algebra $A$ is said to be total if the only positive element of $A$ on which every state in the family vanishes, is zero.
\begin{lem}\label{total4} The family $\Phi(\tilde{\pi})$ is total on
 $\tilde{\pi}(C^*(\Sigma))$. Furthermore, the familiy $\Phi'$ is total on $C^*(\Sigma)$.
\end{lem}
\begin{proof} Suppose the family $\Phi(\tilde{\pi})$ vanishes on $a\geq 0$.
By the comment preceding this lemma it follows that $\sqrt{a}$ is in the right multiplicative domain of every state in $\Phi(\tilde{\pi})$. We consider the closed ideal $J$ generated by $a$, which by the functional calculus coincides
 with the closed ideal generated by $\sqrt{a}$. Note that this ideal is the closed linear span of elements having the form $fu^i agu^j= f(u^ia u^{i*})(u^i gu^{j})$ for functions $f, g \in C(X_{\pi})$. 
Clearly $f$ belongs to the right multiplicative domain of every member of $\Phi(\tilde{\pi})$ by Theorem~\ref{choi4}. To see that $u^i a u^{i*}$ does as well, note firstly that if $\varphi \in \Phi(\tilde{\pi})$ is a pure state extension of the point evaluation in $x \in X_{\pi}$ on $C(X_{\pi})$, then $\varphi \circ \Ad u^i \in \Phi(\tilde{\pi})$ since it is a pure state extension of the point evaluation in $\sigma_{\pi}^{-i}(x)$ on $C(X_{\pi})$ and hence $\varphi (u^i a u^{i*}) = 0$. As $u^i \sqrt{a}u^{i*}$ is the positive square root of $u^i a u^{i*}$ it follows again by the comment preceeding this lemma that $u^i \sqrt{a}u^{i*}$ is in the right multiplicative domain of every element of $\Phi(\tilde{\pi})$ and this clearly implies that $u^i a u^{*i}$ is as well. Hence for every $\varphi \in \Phi(\tilde{\pi})$ we have
\[\varphi(f(u^ia u^{i*})(u^i gu^{j})) = \varphi (f) \varphi(u^i a u^{*i}) \varphi(u^i g u^{j}) = 0\]
since $\varphi (u^i a u^{i*}) = 0$. We conclude that every member of $\Phi(\tilde{\pi})$ vanishes on the whole ideal $J$.
We want to deduce that $J$ is the zero ideal, so assume for a contradiction that it is not. Then there exists a positive element $b \in C^*(\Sigma_{\pi})$ such that $0 \neq \hat{\pi}(b) \in J$. Since $\ker(\hat{\pi})$ is a closed ideal of $C^*(\Sigma_{\pi})$, we know by Proposition~\ref{idbesk4} that it is the intersection of the kernels of a certain family of irreducible representations. So there must be at least one of the representations determining $\ker(\hat{\pi})$ as in Proposition~\ref{idbesk4}, $\tilde{\pi}_{\epsilon}$ say, for which $\tilde{\pi}_{\epsilon}(b) \neq 0$. As $\ker(\hat{\pi}) \subseteq \ker(\tilde{\pi}_{\epsilon})$, we may well-define an irreducible representation $\bar{\pi}$ of $\tilde{\pi}(C^*(\Sigma)) = \hat{\pi} (C^*(\Sigma_{\pi}))$ by $\bar{\pi}(\hat{\pi} (a)) = \tilde{\pi}_{\epsilon} (a)$ for $a \in C^*(\Sigma_{\pi})$.
Then the functional $\bar{\varphi}$ defined by $\bar{\varphi} (\hat{\pi}(a)) = (\bar{\pi}(\hat{\pi}(a)) \xi_{\epsilon}, \xi_{\epsilon}) = (\tilde{\pi}_{\epsilon} (a) \xi_{\epsilon}, \xi_{\epsilon}) = \varphi_{\epsilon} (a)$ is a pure state acting as a point evaluation on the embedded copy of $C(X_{\pi})$ in $\tilde{\pi}(C^*(\Sigma))$. Hence $\bar{\varphi} \in \Phi(\tilde{\pi})$ and by the above $\bar{\varphi} (\hat{\pi} (b)) = 0$ as $\hat{\pi}(b) \in J$. Similarly to above, one easily concludes that $\bar{\varphi} (\hat{\pi} (\delta_{\pi}^i b \delta_{\pi}^{*i})) = \varphi_{\epsilon} (\delta_{\pi}^i b \delta_{\pi}^{*i}) = 0$ for all integers $i$.
Writing this out, using that $b\geq 0$, we get $0 = \varphi_{\epsilon}(\delta_{\pi}^i b \delta_{\pi}^{*i}) = (\tilde{\pi}_{\epsilon}(\delta_{\pi}^i b \delta_{\pi}^{*i}) \xi_{\epsilon}, \xi_{\epsilon}) = (\tilde{\pi}_{\epsilon}(\sqrt{b}\delta_{\pi}^{*i}) \xi_{\epsilon}, \tilde{\pi}_{\epsilon}(\sqrt{b} \delta_{\pi}^{*i}) \xi_{\epsilon}) = \|\tilde{\pi}_{\epsilon}(\sqrt{b}\delta_{\pi}^{*i}) \xi_{\epsilon}\|^2$. Since $\tilde{\pi}_{\epsilon}$ is a representation of the kind described in the passage following Corollary~\ref{bildkp4}, we know that the closed linear span of the set $\{\tilde{\pi}_{\epsilon} (\delta_{\pi}^i) \xi_{\epsilon}\}_{i \in \mathbb{Z}}$ is the whole underlying Hilbert space $H_{\epsilon}$, whence the above equality implies that  $\tilde{\pi}_{\epsilon} (\sqrt{b}) = 0$ and thus finally $\tilde{\pi}_{\epsilon} (b)= 0$. This is a contradiction. Hence $J$ was the zero ideal after all. In particular, $a = 0$. 

Next, to see that $\Phi'$ is total on $C^*(\Sigma)$ we suppose it vanishes on a positive element $a \in C^*(\Sigma)$. Let $x \in \Per^{\infty}(\sigma)$. Since $\varphi_x$ is the unique pure state extension, and hence the unique state extension, of $\mu_x$, we have
\[
E(a)(x)= \mu_x\circ E(a) = \varphi_x(a) =0.
\]
Now let $y \in \PIP(\sigma)$. Since the set $\{\varphi_{y,t} : t \in \mathbb{T}\}$ exhausts all pure state extension of $\mu_y$ to $C^*(\Sigma)$ it follows that if all members vanish at $a$, its state extension $\mu_y \circ E$ vanishes at $a$, too. Hence by Lemma~\ref{untat4} $E(a)$ vanishes on a dense
 subset of $X$, and thus $E(a) = 0$. As $E$ is faithful, this 
implies that $a = 0$.
%The proof of the fact that $\Phi_{\pi}'$ is total on $C^*(\Sigma_{\pi})$ transfers verbatim from above.
\end{proof}
Denote by $\Gamma$ the spectrum of $C(X)'$ and by $\Gamma(\tilde{\pi})$ the spectrum of $\pi(C(X))'$ for an arbitrary representation $\tilde{\pi} = \pi \times u$ of $C^*(\Sigma)$. The following theorem clarifies the structure of $\Gamma(\tilde{\pi})$. As every $C^*$-algebra has a faithful representation it also determines $\Gamma$.
\begin{thm}\label{comcom4}
Let $\tilde{\pi} = \pi \times u$ be a representation of $C^*(\Sigma)$.
\begin{enumerate}
\item  $\pi(C(X))'$ is a commutative $C^*$-subalgebra of $\tilde{\pi}(C^*(\Sigma))$
 (necessarily maximal abelian), invariant under $\Ad u$ and its inverse. In particular, $C(X)'$ is a maximal abelian
 $C^*$-subalgebra of $C^*(\Sigma)$, invariant under $\Ad \delta$ and its inverse. Moreover, the latter is the closure of its
 algebraic part, i.e., of the set of generalized polynomials in $C(X)'$.
\item The spectrum $\Gamma(\tilde{\pi})$ consists of the restrictions of the
 set $\Phi(\tilde{\pi})$ of pure state extensions.
\end{enumerate}
\end{thm}
\begin{proof}
Clearly, $\pi(C(X))'$ is a $C^*$-subalgebra of $\tilde{\pi} (C^*(\Sigma))$.
Take two elements $\tilde{\pi}(a)$ and $\tilde{\pi}(b)$ in
 $\pi(C(X))'$. Then by Proposition~\ref{commbesk4} (ii), $\tilde{\pi}_{x_{\alpha}}(ab - ba) = \tilde{\pi}_{y_{\beta}, t_{\gamma}} (ab-ba) = 0$ for all $\alpha, \beta, \gamma$ determining the kernel of $\tilde{\pi}$ as in Proposition~\ref{idbesk4}. Hence $\pi(C(X))'$ is indeed commutative and clearly it is maximal abelian. Invariance of $\pi(C(X))'$ under $\Ad u$ and its inverse follows readily. The corresponding statement about $C(X)'$ and $\Ad \delta$ and its inverse holds since $C^* (\Sigma)$ has a faithful representation, but can also be obtained by an explicit calculation using Proposition~\ref{commbesk4} (i) together with Proposition~\ref{cesar4}. The last statement of assertion (i) also follows from the characterization of $C(X)'$ in Proposition~\ref{commbesk4} (i) combined with Proposition~\ref{cesar4}.
For the assertion (ii), we may assume that $\tilde{\pi} = \hat{\pi}$, and hence that $C^*(\Sigma) = C^*(\Sigma_{\pi})$ with $X = X_{\pi}$ and $\sigma = \sigma_{\pi}$. Since every element of $\Gamma(\tilde{\pi})$ is the restriction to $\pi(C(X))'$ of at least one element of $\Phi(\tilde{\pi})$, it is sufficient to show that the restriction of each pure
 state in $\Phi(\tilde{\pi})$ to $\pi(C(X))'$ is in $\Gamma(\tilde{\pi})$. This is equivalent to proving that the restriction of every element of $\Phi(\tilde{\pi})$ to $\pi(C(X))'$ is multiplicative. For an aperiodic
point $x \in X$, recall that there is a unique pure state extension of $\mu_x$ from $C(X)$ to $C^*(\Sigma)$, as mentioned in the passage following Corollary~\ref{bildkp4}. This implies that there is a unique pure state extension of $\mu_x$ to $\tilde{\pi}(C^*(\Sigma))$, since if there were two different ones, $\psi_1$ and $\psi_2$, say, then $\psi_1 \circ \tilde{\pi}$ and $\psi_2 \circ \tilde{\pi}$ would be two different pure state extensions of $\mu_x$ to $C^*(\Sigma)$,  which is a contradiction. It follows that there is a unique pure state extension of $\mu_x$ to $\pi(C(X))'$, namely the restriction to this subalgebra of the pure state extension of $\mu_x$ to $\tilde{\pi}(C^*(\Sigma))$. The pure states of a commutative $C^*$-algebras are precisely its characters, so this restriction is indeed multiplicative. Now
 take a periodic point $y \in X$ and consider a pure state extension
$\varphi'_y$ of $\mu_y$ to $\tilde{\pi}(C^*(\Sigma))$. Then the pure
 state $\varphi'_y\circ \tilde{\pi}$ is a pure state extension of $\mu_y$ to
 $C^*(\Sigma)$, hence it is of the form $\varphi_{y,t}$ for some $t \in \mathbb{T}
 $. Then $\varphi_{y,t}$ vanishes on $I = \ker(\tilde{\pi})$ and it follows that $\tilde{\pi}_{y,t}$ does as well.
 To see this, let $a \in I$ be a positive element. Then, as $\varphi_{y,t}$ vanishes on $I$, we have that $0 = \varphi_{y,t} (\delta^i a \delta^{*i}) = (\tilde{\pi}_{y,t}(\sqrt{a}\delta^{*i}) \xi_{y,t}, \tilde{\pi}_{y,t}(\sqrt{a}\delta^{*i}) \xi_{y,t})$ and hence $\tilde{\pi}_{y,t}(\sqrt{a}) = 0$, as $H_{y,t}$ is the closed linear span of the set $\{\tilde{\pi}_{y,t} (\delta^i) \xi_{y,t}\}_{i \in \mathbb{Z}}$. We conclude that $\tilde{\pi}_{y,t}$ vanishes on $I$ since a closed ideal is generated by its positive part. This implies that there are parameters $\beta, \gamma$ appearing in the description of $I$ as in Proposition~\ref{idbesk4} such that $y_{\beta} = y$ and $t_{\gamma} = t$. Now, for an element
 $\tilde{\pi}(a)$ in $\pi(C(X))'$ we first see that
\[
\varphi'_y(\tilde{\pi}(a)) = \varphi_{y,t}(a) =
 (\tilde{\pi}_{y,t}(a)\xi_{y,t}, \xi_{y,t}).
\]
By (ii) of Proposition~\ref{commbesk4} we can replace $a$ by a function $f \in C(X)$. It follows that the rightmost side of the above equalities becomes
 simply
$f(y)$. Therefore, if for two elements $\tilde{\pi}(a)$ and
 $\tilde{\pi}(b)$ in $\pi(C(X))'$ we replace $a$ and $b$ with the functions $f$ and
 $g$, we 
have that
\begin{eqnarray*}
\varphi'_y(\tilde{\pi}(a)\tilde{\pi}(b))&=& \varphi'_y\circ
 \tilde{\pi}(ab) = \varphi'_y \circ \tilde{\pi} (fg) \\
    &= & f(y)g(y) =
 \varphi'_y(\tilde{\pi}(a))\varphi'_y(\tilde{\pi}(b)).
\end{eqnarray*}
Hence restrictions of pure state extensions to $\pi(C(X))'$ always
 induce characters on it. This completes the proof.
\end{proof}
We now introduce some notation. A character in $\Gamma(\tilde{\pi})$ is denoted by $\gamma(x)$ if it is the restriction to $\pi(C(X))'$ of some $\psi_x \in \Phi(\tilde{\pi})$ such that $\psi_x \circ \hat{\pi} = \varphi_x$ for $x \in \Per^{\infty}(\sigma_{\pi})$. Similarly, we denote by $\gamma(y,t)$ a character in $\Gamma(\tilde{\pi})$ that is the restriction to $\pi(C(X))'$ of some $\psi_{y,t} \in \Phi(\tilde{\pi})$ such that $\psi_{y,t} \circ \hat{\pi} = \varphi_{y,t}$ for $y \in \Per(\sigma_{\pi})$ and $t \in \mathbb{T}$.
Clearly every element of $\Gamma(\tilde{\pi})$ is of this form.
Note that in general it may happen that $\gamma(y,s) = \gamma(y,t)$ even though $s \neq t$. If $\tilde{\pi}$ is faithful and $\Sigma$ is topologically free, for example, we know by Theorem~\ref{triquivc4} that $\pi(C(X))' = \pi(C(X))$ and hence $\gamma(y,t)$ is independent of $t$.
A slight subtlety occurs as for $y \in \Per(\sigma_{\pi})$ it is not necessarily so that $\gamma(y,t)$ appears in $\Gamma(\tilde{\pi})$ for each $t \in \mathbb{T}$. 
In the following lemma we determine the $\gamma(y,t)$ that appear.
\begin{lem}\label{vilkafinns4}
$\Gamma(\tilde{\pi}) = \{\gamma(x) : x \in \Per^{\infty}(\sigma_{\pi})\} \cup \{\gamma(y,t) : \ker(\hat{\pi}) \subseteq \ker(\tilde{\pi}_{y,t})\}$. For every $y \in \Per (\sigma_{\pi})$ there is at least one $t \in \mathbb{T}$ such that $\gamma(y,t) \in \Gamma(\tilde{\pi})$.
\end{lem}
\begin{proof}
That $\gamma(x)$ appears in $\Gamma(\tilde{\pi})$ for every $x \in \Per^{\infty}(\sigma_{\pi})$ follows from the fact that $\mu_x$ has a unique pure state extension to $C^*(\Sigma_{\pi})$, namely $\varphi_x$. The last statement of the lemma follows since every $\mu_y$ has a pure state extension to $\hat{\pi}(C^*(\Sigma_{\pi}))$, hence its composition with $\hat{\pi}$ must be of the form $\varphi_{y,t}$ for some $t \in \mathbb{T}$.
If $\gamma(y,t) \in \Gamma(\tilde{\pi})$, then it has a pure state extension $\psi_{y,t}$ to $\hat{\pi}(C^*(\Sigma_{\pi}))$ such that $\psi_{y,t} \circ \hat{\pi} = \varphi_{y,t}$, whence $\varphi_{y,t}$ vanishes on $\ker({\hat{\pi}})$. As in the proof of Theorem~\ref{comcom4}, it follows readily that $\ker(\hat{\pi}) \subseteq \ker(\tilde{\pi}_{y,t})$. Conversely, if $\ker(\hat{\pi}) \subseteq \ker(\tilde{\pi}_{y,t})$ we can well-define a pure state extension $\psi_{y,t}$ of $\mu_y$ to $\hat{\pi}(C^*(\Sigma_{\pi}))$ by $\psi_{y,t} (\hat{\pi} (a)) = \varphi_{y,t} (a)$ for $a \in C^*(\Sigma_{\pi})$. In this notation, the restriction of $\psi_{y,t}$ to $\pi(C(X))'$ is $\gamma(y,t)$.
\end{proof}
When using this notation, the homeomorphism $\sigma(\tilde{\pi})$ of $\Gamma (\tilde{\pi})$ induced by the automorphism $\Ad u$ of $\pi(C(X))'$ is such that $\sigma(\tilde{\pi}) (\gamma(x)) = \gamma(\sigma_{\pi}(x))$ for $x \in \Per^{\infty}(\sigma_{\pi})$ and $\sigma(\tilde{\pi}) (\gamma(y,t)) = \gamma(\sigma_{\pi}(y),t)$ for $y \in \Per (\sigma_{\pi})$ and $t \in \mathbb{T}$. To see this, note that if, for example, $\psi_{y,t} \circ \hat{\pi} = \varphi_{y,t}$, then $\psi_{y,t} \circ \Ad u \circ \hat{\pi} = \varphi_{\sigma_{\pi}^{-1}(y),t}$. Since $\pi(C(X))'$ is invariant under $\Ad u$ and its inverse by Theorem~\ref{comcom4} (i) and $\psi_{y,t}$ extends $\gamma(y,t)$, it follows that $\psi_{y,t} \circ \Ad u$ extends $\gamma(y,t) \circ \Ad u$. We use the notation $\tilde{\Sigma} = (\Gamma(\tilde{\pi}), \sigma(\tilde{\pi}))$ for the corresponding dynamical system. Similarly, we denote by $\tilde{\sigma}$ the homeomorphism of $\Gamma$ induced by the restriction of the automorphism $\Ad \delta$ to $C(X)'$.
Note that $\gamma(x) \in \Per^{\infty} (\sigma(\tilde{\pi}))$ for every $x \in \Per^{\infty} (\sigma_{\pi})$ and that $\gamma(y,t) \in \Per_k (\sigma(\tilde{\pi}))$ implies that $y \in \Per_k (\sigma_{\pi})$. 

Define $\Phi'(\tilde{\pi})$ to be the set of all pure state extensions of elements of \[\Per^{\infty}(\sigma(\tilde{\pi})) \cup \PIP(\sigma(\tilde{\pi}))\] to $\tilde{\pi}(C^*(\Sigma))$.
It is important to know which elements of $\Gamma(\tilde{\pi})$ have unique pure state extensions to elements of $\Phi(\tilde{\pi})$. The following proposition sheds some light upon this question.
\begin{pro}\label{unutv4}
The elements of $\Per^{\infty}(\sigma(\tilde{\pi})) \cup \PIP(\sigma(\tilde{\pi}))$ all have a unique extension to an element of $\Phi(\tilde{\pi})$.
\end{pro}
\begin{proof}
As mentioned above, points in $\Per^{\infty}(\sigma(\tilde{\pi}))$ must be of the form $\gamma(x)$ where $x \in \Per^{\infty} (\sigma_{\pi})$. By definition, there is a $\psi_x \in \Phi(\tilde{\pi})$ such that $\psi_x \circ \hat{\pi} = \varphi_x$. 
As $\varphi_x$ is the unique pure state extension of $\mu_x$ to $C^* (\Sigma_{\pi})$ we see that, if there were another extension $\psi \in \Phi(\tilde{\pi})$, then $\psi \circ \hat{\pi} = \psi_x \circ \hat{\pi}$, and thus $\psi = \psi_x$. 
Now consider the set $\Per_k(\sigma(\tilde{\pi}))^0$ for some positive integer $k$ and suppose that the pure
 states $\psi_{y,t}$ and $\psi_{y,s}$ induce the same point there, i.e., $\gamma(y,t) = \gamma(y,s) \in \Per_k(\sigma(\tilde{\pi}))^0$. Let $F \in C(\Gamma(\tilde{\pi}))$ be a continuous function
 whose support is contained in that interior such that $F(\gamma(y,t))$
 is 
not zero. Then the element $Fu^k$ is an element of $\pi(C(X))'$. To see this, let $\pi(g) \in \pi(C(X))$ and note that the condition on the support of $F$ implies that $F u^k \pi(g) = F u^k \pi(g) u^{*k} u^k = F \pi(g) \circ \sigma(\tilde{\pi})^{-k} u^k = \pi(g) F u^k$.
Since $F$ belongs to the right multiplicative domain for both $\psi_{y,t}$ and $\psi_{y,s}$ by Theorem~\ref{choi4}, and since $\Per(y) = \Per(\gamma(y,t)) = k$, we
have
\begin {eqnarray*}
Fu^k (\gamma(y,t))  &=& \psi_{y,t}(Fu^k) = \psi_{y,t}(F)\psi_{y,t}(u^k) =
 F(\gamma(y,t)) \psi_{y,t} \circ \hat{\pi} (\delta_{\pi}^k)\\ &= & F(\gamma(y,t)) \varphi_{y,t} (\delta_{\pi}^k)
             = F(\gamma(y,t))t.
\end{eqnarray*}
On the other hand, we conclude in the same fashion that $Fu^k (\gamma(y,s)) = F(\gamma(y,s))s$
and hence that $t = s$. Thus $\psi_{y,t} = \psi_{y,s}$. When the $y$-components are different, say as $y_1$ and $y_2$, we can easily separate $\psi_{y_1,t}$ and $\psi_{y_2,s}$ by functions in $C(X_{\pi})$. Thus all elements of $\Per_k(\sigma(\tilde{\pi}))^0$ are uniquely extended to elements of $\Phi(\tilde{\pi})$.
\end{proof}
Note that continuity of restriction maps implies that \[\{\gamma(y,t) \in \Gamma(\tilde{\pi}) : y \in \PIP(\sigma_{\pi})\} \subseteq \PIP(\sigma(\tilde{\pi})),\]
and clearly $\{\gamma(x) \in \Gamma(\tilde{\pi}) : x \in \Per^{\infty} (\sigma_{\pi})\} = \Per^{\infty} (\sigma(\tilde{\pi}))$.
\section{Ideal intersection property of $C(X)'$ and $\pi(C(X))'$}
In the theory of $C^*$-crossed products, an important direction of research is 
the investigation of how the structure of closed ideals of a crossed product reflects
 its building block $C^*$-subalgebra, which in our case is $C(X)$. 
Theorem~\ref{triquivc4} sheds some light upon this question.
It tells us that all non-zero closed ideals have non-zero intersection with $C(X)$ 
precisely when the system is topologically free. Hence for
 a dynamical system that is not topologically free, for example rational
 rotation of the unit circle, there always exists a non-zero closed ideal whose intersection with $C(X)$ is zero.

In this section we analyse the situation for general
 dynamical systems. Since Theorem~\ref{triquivc4} tells us that $C(X)$ has non-zero intersection with every non-zero closed ideal precisely when $C(X) = C(X)'$, and since $C(X) \subseteq C(X)'$ in general, it appears natural to investigate what ideal intersection properties $C(X)'$ has for an arbitrary system.
In \cite[Theorem 6.1]{SSD24} it is proved that in an algebraic crossed product the commutant of the subalgebra corresponding to $C(X) \subseteq C^* (\Sigma)$ always has non-zero intersection with every non-zero ideal of the crossed product. In \cite[Theorem 3.1]{SSD34} a generalization of that result with a more elementary proof is provided.

We make the following definition.

\begin{defn}\label{defintpro4}Let $A$ be a unital $C^*$-algebra and $B$ be a commutative
 $C^*$-subalgebra of $A$ with the same unit.

\begin{enumerate}
 \item $B$ is said to have the intersection property for closed ideals of
 $A$ if for any non-zero closed ideal $I$ of $A$, $I \cap B \neq \{0\}$
\item $B$ is said to have the intersection property for ideals of $A$ if
 for any non-zero ideal $J$ of $A$, not necessarily closed or self-adjoint, $J \cap B \neq \{0\}$.
\end{enumerate}
\end{defn}

\begin{pro}\label{sammaskar4}
The above two properties for $B$ are equivalent.
\end{pro}
\begin{proof} It is enough to show that (i) implies (ii). We let $Y$ be the
 spectrum of $B$ and identify $B$ with $C(Y)$. Let $J$ be a non-zero ideal of $A$. 
By assumption, $B\cap \bar{J} \neq \{0\}$. Since $B \cap \bar{J}$ is then a non-zero closed ideal of $B$, we may, under the identification $B \cong C(Y)$, write $B \cap \bar{J} = \ker(Z)$ for some closed subset $Z \subsetneq Y$. Take a point $p \in Y \setminus Z$ and a positive function $g \in C(Y)$ vanishing on $Z$ and such that $g(p) = 1$.
Let $f \in C_c (\mathbb{R})$ be such that $f$ vanishes on $(-\infty, \frac{1}{2}]$ and $f(1) = 1$. Then $f_{\upharpoonright (0, \infty)} \in C_c ((0,\infty))$ and by the characterization of the minimal dense ideal  (also known as the Pedersen ideal) $P_{\bar{J}}$ of $\bar{J}$ in  \cite[Theorem 5.6.1]{P34} (see also \cite{P14},\cite{P24}), we have that $f (g) = f \circ g \in P_{\bar{J}}$. Since $P_{\bar{J}} \subseteq J$, this implies that $f \circ g  \in J$ and since clearly also $f \circ g \in B$ we are done.
\end{proof}
\begin{thm}\label{intpro4} Given a representation $\tilde{\pi}$, if the family
 $\Phi'(\tilde{\pi})$ is total on $\tilde{\pi}(C^*(\Sigma))$, then the algebra
 $\pi(C(X))'$
has the intersection property for ideals in $\tilde{\pi}(C^*(\Sigma))$.
% Consequently, $C(X)'$ and $C(X_{\pi})'$ have the intersection property for ideals in $C^*(\Sigma)$
%and $C^*(\Sigma_{\pi})$, respectively.
\end{thm}
\begin{proof} By Proposition~\ref{sammaskar4}, it is enough to consider closed ideals. Suppose there exists a closed ideal $I$ that has trivial
 intersection with $\pi(C(X))'$. We shall show that \(I = \{0\}\). Let
 $q$ be the 
quotient map from $\tilde{\pi}(C^*(\Sigma))= \hat{\pi}(C^*(\Sigma_{\pi}))$
 to the quotient algebra $\tilde{\pi}(C^*(\Sigma))/I$. Then it induces an
 isomorphism 
between $\pi(C(X))'$ 
and $q(\pi(C(X))')$. This implies that every pure state $\varphi \in \Gamma (\tilde{\pi})$ on $\pi(C(X))'$ is of the form ${\xi \circ q}_{\upharpoonright \pi(C(X))'} $, where $\xi$ is a pure state of $q(\pi(C(X))')$. Let $\tilde{\xi}$ be a pure state extension of $\xi$ to $q(\tilde{\pi} (C^* (\Sigma))$. Clearly, $\tilde{\varphi} = \tilde{\xi} \circ q$ is a pure state extension of $\varphi$ to $\tilde{\pi}(C^*(\Sigma))$.
Since pure state extensions of elements of $\Per^{\infty}(\tilde{\sigma}) \cup \PIP(\tilde{\sigma}) \subseteq \Gamma(\tilde{\pi})$ to $\tilde{\pi}(C^*(\Sigma))$ are unique by Proposition~\ref{unutv4}, it follows that every element of $\Phi' (\tilde{\pi})$ factors through $q$ and hence vanishes on $I$. As $\Phi' (\tilde{\pi})$ was assumed to be total, we conclude that $I = \{0\}$.
%To conclude that the statements about $C^*(\Sigma)$ and $C^*(\Sigma_{\pi})$ hold, we just consider faithful representations of these algebras and apply Lemma~\ref{total4}.
\end{proof}
\begin{cor}\label{intprocol4}
$C(X)'$ has the intersection property for ideals of $C^* (\Sigma)$.
\end{cor}
\begin{proof}
Consider a faithful representation of $C^* (\Sigma)$ and apply Lemma~\ref{total4}.
\end{proof}
The following example shows that it can happen that $\tilde{\pi}(C(X)') \subsetneq \pi(C(X))'$ and that the former does not necessarily have the intersection property for ideals.
\begin{ex}
We consider again the dynamical system mentioned in Section~\ref{notprel4}. 
Namely, let $X = [0,1]\times [-1,1] \subseteq \mathbb{R}^2$ be endowed with the standard topology and let $\sigma$ be the homeomorphism of $X$ defined as reflection in the $x$-axis. Using the notation $\Sigma_{[0,1]} = ([0,1], \sigma_{\upharpoonright [0,1]})$ and noting that $\sigma_{\upharpoonright [0.1]}$ is just the identity homeomorphism,  we consider the map from $C^* (\Sigma)$ to $C^* (\Sigma_{[0,1]})$ defined on the algebraic part of $C^* (\Sigma)$ as $\sum_{n} f_n \delta^n \mapsto \sum_{n} {f_n}_{\upharpoonright [0,1]} \delta_{[0,1]}^n$, where $\delta_{[0,1]}$ denotes the canonical unitary element of $C^* (\Sigma_{[0,1]})$. It cleary extends by continuity to a surjective homomorphism $\tilde{\pi} = \pi \times u : C^*(\Sigma) \rightarrow C^*(\Sigma_{[0,1]})$. By Proposition~\ref{commbesk4} (i), $C(X)'$ is the $C^*$-subalgebra of $C^*(\Sigma)$ generated by $C(X)$ and $\delta^2$. We also note that $C^*(\Sigma_{[0,1]}) \cong C([0,1] \times \mathbb{T})$. As $\pi(C(X)) = C([0,1])$, and $C^*(\Sigma_{[0,1]})$ is commutative, we see that $\pi(C(X))' = C^*(\Sigma_{[0,1]})$, while $\tilde{\pi}(C(X)')$ only contains elements $a \in C^*(\Sigma_{[0,1]})$ with $a(j) = 0$ for odd integers $j$. Hence $\tilde{\pi}(C(X)') \subsetneq \pi(C(X))'$.
To see that $\tilde{\pi}(C(X)')$ does not have the intersection property for ideals, consider the ideal $I \subseteq C([0,1]\times \mathbb{T})$ consisting of all functions that vanish on $[0,1] \times C$, where $C \subseteq \mathbb{T}$ is the closed upper halfcircle. As $\tilde{\pi}(C(X)')$ is identified with the $C^*$-subalgebra of $C([0,1] \times \mathbb{T})$ generated by $C([0,1])$ and $z^2$ it is clear that $\tilde{\pi}(C(X)') \cap I = \{0\}$.
\end{ex}

\section{Intermediate subalgebras}\label{intsub4}
As we now know that $C(X)'$  always has the intersection property for ideals in $C^*(\Sigma)$, while $C(X)$ has it if and only if $C(X) = C(X)'$ by Theorem~\ref{triquivc4}, we shall consider $C^*$-subalgebras $B$ with $C(X) \subseteq B \subseteq C(X)'$ and in particular investigate their intersection properties with ideals. We call such a $C^*$-subalgebra an intermediate subalgebra. In \cite[Section 5]{SSD34} intermediate subalgebras are studied in the context of an algebraic crossed product.
The following proposition gives an abstract characterization of intermediate subalgebras which have the intersection property, and consequently of those who do not, in terms of the relation of their spectra to the spectrum of $C(X)'$.
\begin{pro}\label{specint4}
Suppose $B \subseteq C^*(\Sigma)$ is a $C^*$-subalgebra such that $C(X) \subseteq B \subseteq C(X)'$ and denote by $\Delta$ and $\Gamma$ the spectrum of $B$ and $C(X)'$, respectively. Then $B$ has the intersection property for ideals if and only if for every proper closed subset $S \subseteq \Gamma$ that is invariant under $\widetilde{\sigma}$ and its inverse, the restriction of $S$ to $B$ is a proper subset of $\Delta$.
\end{pro}
\begin{proof}
By Proposition~\ref{sammaskar4}, it is enough to consider closed ideals.
Let $S$ be a proper closed subset of $\Gamma$ that is invariant under $\tilde{\sigma}$ and its inverse. Let $I$ be the closed ideal of $C^*(\Sigma)$ generated by the hull $h(S)$ of $S$, defined by $h(S) = \{a \in C(X)' : s(a) = 0 \textup{ for all } s \in S\}$.
We first assert that $I \cap C(X)' = h(S)$. Note that $I$ is the closed linear span of elements of the form $f\delta^i F g\delta^j =
 f (\delta^i Fg \delta^{* i})\ \delta^{i+j}$, where $F \in h(S)$. As $Fg \in C(X)'$, we can rewrite this as $f (\delta^i Fg \delta^{* i})\ \delta^{i+j} = f  (Fg) \circ \tilde{\sigma}^{-i} \delta^{i+j}$.
  Let $\varphi_x$ and $\varphi_{y,t}$ be pure
 state extensions of points $\gamma(x)$ and $\gamma(y,t)$ in $S$.
 Since the element $f  (Fg) \circ \tilde{\sigma}^{-i}$ belongs to $C(X)'$, and hence to the right  multiplicative
 domain for $\varphi_x$ and $\varphi_{y,t}$ by Theorem~\ref{choi4}, and to $h(S)$ by invariance of $S$ under $\tilde{\sigma}$ and its inverse, we have
\[
\varphi_x(f\delta^i F g\delta^j) =
 \varphi_x(f  (Fg) \circ \tilde{\sigma}^{-i})\varphi_x(\delta^{i+j}) = 0,
\]
and also
\[
\varphi_{y,t}(f\delta^i F g\delta^j)
 =\varphi_{y,t}(f  (Fg) \circ \tilde{\sigma}^{-i})\varphi_{y,t}(\delta^{i+j}) = 0.
\]
Therefore, all pure states $\varphi_x$ and $\varphi_{y,t}$ vanish on
 $I$, and thus $I \cap C(X)' \subseteq h(S)$ and we conclude that $I \cap C(X)' = h(S)$. Denoting by $r_B : \Gamma \rightarrow \Delta$ the restriction map, we see that $I \cap B =
 h(r_B(S))$. So if $B$ has the intersection property, then $h(r_B(S)) \neq \{0\}$, which implies $r_B (S) \subsetneq \Delta$. On the other hand,
 suppose that $B$ does not have the intersection property and hence that there exists a nonzero closed ideal $I$ that has trivial intersection with $B$. Since $C(X)'$ itself has the intersection property by Corollary~\ref{intprocol4}, and since it is easy to see that $I \cap C(X)'$ is a closed ideal of $C(X)'$ invariant under $\Ad \delta$ and its inverse, we may write $I \cap C(X)' = h(S)$ for some proper closed subset $S$ of $\Gamma$ that is invariant under $\tilde{\sigma}$ and its inverse. Since $B \cap I = \{0\}$ we see that $\{0\} = B \cap I \cap C(X)' = B \cap h(S) = h(r_B (S))$, whence $r_B (S) = \Delta$.
\end{proof}
Combining this result with Proposition~\ref{commbesk4}, Theorem~\ref{comcom4} and Corollary~\ref{intprocol4} we can provide the following alternative proof of a refined version of Theorem~\ref{triquivc4}.
\begin{thm}\label{triquivcny4}
For a topological dynamical system $\Sigma$, the following statements are equivalent.
\begin{enumerate}
 \item $\Sigma$ is topologically free;
\item $I\cap C(X) \neq 0$ for every non-zero ideal $I$ of $C^*(\Sigma)$;
\item $C(X)$ is a maximal abelian $C^*$-subalgebra of $C^*(\Sigma)$.
\end{enumerate}
\end{thm}
\begin{proof}
Since $C(X)'$ is maximal abelian by Theorem~\ref{comcom4} (i), equivalence of (i) and (iii) is an immediate consequence of Proposition~\ref{commbesk4} (i).
The implication (iii) $\Rightarrow$ (ii) follows from Corollary~\ref{intprocol4} since $C(X) \subseteq C(X)'$.
To prove (ii) $\Rightarrow$ (i), suppose that $\Per^{\infty}(\sigma)$ is not dense in $X$. Then $C(X) \subsetneq C(X)'$ by the above, and by Theorem~\ref{commbesk4}(i) it follows that there exists a positive $k$ for which $\Per^k (\sigma)^0 \neq \emptyset$.
Let $S = \overline{\{\gamma(x) : x \in \Per^{\infty}(\sigma)\} \cup \{\gamma(y, 1) : y \in \Per (\sigma)\}}$. It is easy to see that $S$ is invariant under $\tilde{\sigma}$ and its inverse, and clearly its restriction to $C(X)$ coincides with $X$. We will show that $S$ is a proper subset of $\Gamma$. Let $y \in \Per^k (\sigma)^0$. Then there are positive integers $l, r$ such that $\Per (y) = l$ and $k = l \cdot r$. Let $s$ be a complex number such that $s^r = i$ and let $f \in C(X)$ be real-valued and such that $\supp(f) \subseteq \Per^{k} (\sigma)^0$ and $f(y) = 1$. Then by Proposition~\ref{commbesk4} (i) it follows that $f \delta^k \in C(X)'$. To see that $\gamma(y,s) \notin S$, note that $\gamma(x) (f \delta^k)$ and $\gamma(y,1) (f \delta^k)$ are real valued for all $x \in \Per^{\infty} (\sigma)$ and $y \in \Per (\sigma)$, while $\gamma(y,s) (f \delta^k) = s^r \cdot f(y) = i \cdot 1$. Hence by Proposition~\ref{specint4}, $C(X)$ does not have the intersection property for ideals.
\end{proof}
Given a dynamical system $\Sigma$, it seems natural to ask which kinds of \emph{properly} intermediate $C^*$-subalgebras, by which we mean $C^*$-subalgebras $B$ with $C(X) \subsetneq B \subsetneq C(X)'$, exist in $C^*(\Sigma)$, regarding the intersection property for ideals.
In~\cite[Theorem 5.4]{SSD34}, a related result is obtained in the case when an algebraic crossed product is associated with a more general dynamical system than the ones considered in this paper.
In our context of $C^*$-crossed products, it turns out that if $C(X) \subsetneq C(X)'$, there always exist $C^*$-subalgebras $B_1, B_2$ with $C(X) \subsetneq B_i \subsetneq C(X)'$ for $i = 1,2$ such that $B_1$ has the intersection property for ideals and $B_2$ does not. The result will follow easily from a number of propositions in which we construct intermediate $C^*$-subalgebras of $C^* (\Sigma)$ for restricted classes of $\Sigma$.
The following two propositions are extensions of~\cite[Proposition 5.2]{SSD34} and~\cite[Proposition 5.3]{SSD34}, respectively, to the context of $C^*$-crossed products.
\begin{pro}\label{ejintp4}
Suppose that there exists an integer $n >0$ such that $\Per^n(\sigma)^0$ contains at least two orbits. Then there exists a $C^*$-subalgebra $B$ of $C^*(\Sigma)$ with $C(X) \subsetneq B \subsetneq C(X)'$ that does not have the intersection property for ideals.
\end{pro}
\begin{proof}
Fix an integer $n >0$ such that $\Per^n(\sigma)^0$ contains at least two orbits. Then it is easy to see that there exist two disjoint open subsets $U_1, U_2$ of $\Per^n(\sigma)^0$ that are invariant under $\sigma$ and $\sigma^{-1}$. Define $B = \{a \in C^*(\Sigma) : \supp(a(k)) \subseteq U_1 \cap \Per^k(\sigma) \textup{ for } k \neq 0\}$. Using the continuity of the projection map $E$ and Proposition~\ref{cesar4} it is easy to check that $B$ is indeed a $C^*$-subalgebra of $C^*(\Sigma)$. By the explicit description of $C(X)'$ in Proposition~\ref{commbesk4}(i) we see that we have the inclusions $C(X) \subseteq B \subseteq C(X)'$. To see that the inclusions are strict, consider non-zero functions $f_{n_1}, f_{n_2} \in C(X)$ with support in $U_1$ and $U_2$, respectively. Then $f_{n_1} \delta^n \in B \setminus C(X)$ and $f_{n_2} \delta^n \in C(X)' \setminus B$ and thus $C(X) \subsetneq B \subsetneq C(X)'$.
We now exhibit a non-zero closed ideal that has trivial intersection with $B$. Let $f_{n_2}$ be as above and let $I = (f_{n_2}-f_{n_2}\delta^n)$, the closed ideal generated by the element $f_{n_2}-f_{n_2}\delta^n$. We first prove that $I \cap C(X) = \{0\}$. For a point $y \in \Per(\sigma)$, consider the representation $\tilde{\pi}_{y,1}$ of $C^*(\Sigma)$ as described in the passage following Corollary~\ref{bildkp4}. 
We see that for $y \notin \Per^n (\sigma)$, $\tilde{\pi}_{y,1}(f_{n_2}) =0$ whence $\tilde{\pi}_{y,1} (f_{n_2}-f_{n_2}\delta^n) =0 $ and hence $\tilde{\pi}_{y,1}$ vanishes on $I$ for such $y$. For $y \in \Per^n(\sigma)$, clearly $\Per(y) | n$. Then $\tilde{\pi}_{y,1} (\delta^n) = \tilde{\pi}_{y,1}(1) = \id$ and thus $\tilde{\pi}_{y,1} (f_{n_2}-f_{n_2}\delta^n) = 0$, so $\tilde{\pi}_{y,1}$ vanishes on $I$ also in this case. For $x \in \Per^{\infty}(\sigma)$, it is clear that $\tilde{\pi}_x (f_{n_2}) = 0$ and thus $\tilde{\pi}_x$ vanishes on $I$. Suppose now that $f \in C(X) \cap I$. Since the family $\{\tilde{\pi}_{x} : x \in \Per^{\infty} (\sigma)\} \cup \{\tilde{\pi}_{y,1} : y \in \Per (\sigma)\}$ is easily seen to be total on $C(X)$, we conclude that $f$ must be zero and thus $I \cap C(X) = \{0\}$. Knowing this, to conclude that $B \cap I = \{0\}$ it suffices to show that for any integer $j \neq 0$ and $a \in B \cap I$ we have $a(j) = 0$. Clearly, $I$ can be described as the closure of the set  \[\{(\sum_{i} g_i \delta^i)(f_{n_2} - f_{n_2}\delta^n)(\sum_{j} h_j \delta^j) : \textup{only finitely many } g_i, h_j \in C(X) \textup{ non-zero}\}.\] The facts that $\supp(f_{n_2}) \subseteq U_2$ and that $U_2$ is invariant under $\sigma$ and $\sigma^{-1}$ imply that the generalized Fourier coefficients of elements of this dense subset of $I$ all have support contained in $U_2$. Continuity of $E$ then implies that for $a \in I$ and $i \in \mathbb{Z}$ we have $\supp(a(i)) \subseteq U_2$. Now suppose $a \in B \cap I$. By the above we know that  $\supp(a(i)) \subseteq U_2$ for every $i \in \mathbb{Z}$ since $a \in I$, but by definition of $B$ we also have that, for $i \neq 0, \supp(a(i)) \cap U_2 = \emptyset$. So the only non-zero generalized Fourier coefficient of $a$ is $a(0)$, meaning that $a \in C(X)$. By the above, however, we know that $I \cap C(X) = \{0\}$ so $B \cap I = \{0\}$ as desired.
\end{proof}
\begin{pro}\label{intp4}
Suppose that there exists an integer $n >0$ such that $\Per^n (\sigma)^0$ contains a point $x_0$ that is not isolated. Then there exists a $C^*$-subalgebra $B$ of $C^*(\Sigma)$ with $C(X) \subsetneq B \subsetneq C(X)'$ that has the intersection property for ideals.
\end{pro}
\begin{proof}
Suppose that a point $x_0 \in X$ is as in the statement of the theorem. Define \[B = \{a \in C(X)' : a(k) (x_0) = 0 \textup{ for } k \neq 0\}.\] Using the continuity of the projection map $E$ and Proposition~\ref{cesar4} it is easy to check that $B$ is indeed a $C^*$-subalgebra of $C^*(\Sigma)$ and clearly $C(X) \subseteq B \subseteq C(X)'$. To show that the inclusions are strict we note that the fact that $x_0$ is not isolated implies that $\Per^n (\sigma)^0$ must be infinite, so we can choose an $x_1 \in \Per^n (\sigma)^0$ and separate $x_0$ from $x_1$ by two open subsets $U_0, U_1 \subseteq \Per^n (\sigma)^0$. Take two functions $f_{n_0}, f_{n_1} \in C(X)$ with support contained in $U_0$ and $U_1$, respectively,  and such that $f_{n_0} (x_0) \neq 0$ and $f_{n_1} (x_1) \neq 0$. Then $f_{n_0} \delta^n \in C(X)' \setminus B$ and $f_{n_1}\delta^n \in B \setminus C(X)$ and thus $C(X) \subsetneq B \subsetneq C(X)'$. 
Now let $I \subseteq C^*(\Sigma)$ be an ideal. We wish to show that $B \cap I \neq \{0\}$. By Corollary~\ref{intprocol4} we know that $C(X)' \cap I \neq \{0\}$. Suppose $0 \neq a \in C(X)' \cap I$. If $a(i)(x_0) = 0$ for every integer $i \neq 0$ then $a \in B \cap I$ and we are done, so suppose there is an $i \neq 0$ such that $a(i)(x_0) \neq 0$. Since $x_0$ is not isolated we can find an $x_1 \neq x_0$ such that $a(i)(x_1) \neq 0$. Separate $x_0$ and $x_1$ by open sets $V_0, V_1$ and choose a function $g \in C(X)$ such that $g(x_1) = 1$ and $\supp(g) \subseteq V_1$. The module property of $E$ (by which we mean that if $a \in C^*(\Sigma)$ and $f, g \in C(X)$ we have $E(fag) = f E(a) g$ implies that $(ga)(i) = g \cdot a(i) \neq 0$. Hence $0 \neq ga \in B \cap I$.
\end{proof}
To analyze the last possibility we need to recall two basic results from the theory of $C^*$-crossed products as introduced in this paper.
The following result is part of~\cite[Proposition 3.5]{T34}.
\begin{pro}\label{enbana4}
Suppose $\Sigma = (X, \sigma)$ is such that $X$ consists of a single $\sigma$-orbit of order $p$, $X = \{x, \sigma(x), \ldots, \sigma^{p-1} (x)\}$, where $p$ is a positive integer. Then $C^* (\Sigma) \cong C(\mathbb{T}, M_p)$, where $M_p$ is the set of $p \times p$-matrices over the complex numbers, via the isomorphism $\widetilde{\rho}_x = \rho_x \times u$, where \[ \rho_{x} (f)=\left( 
\begin{array}{cccc}
f(x) & 0 & \ldots & 0 \\
0 & f \circ \sigma (x) & \ldots & 0 \\
\vdots & \vdots & \ddots & \vdots \\
0 & 0 & \ldots & f \circ \sigma^{p-1} (x)
\end{array} \right)\]
for $f \in C(X)$, and
\[ u (z)=\left( 
\begin{array}{ccccc}
0 & 0 & \ldots & 0 & z \\
1 & 0 & \ldots & 0 & 0 \\
0 & 1 & \ldots & 0&0\\
\vdots & \vdots & \ddots & \vdots & \vdots \\
0 & 0 & \ldots & 1&0
\end{array} \right)\]
for $z \in \mathbb{T}$.
\end{pro}
The following result is easily proved using the definition of a $C^*$-crossed product.
\begin{lem}\label{dirsum4}
Suppose $\Sigma = (X, \sigma)$ is such that $X$ is the union of two disjoint non-empty open (hence closed) subsets $A_1, A_2 \subseteq X$ that are invariant under $\sigma$ and its inverse. Denoting $\Sigma_i = (A_i, {\sigma}_{\upharpoonright A_i})$ for $i = 1, 2$ we then have that $C^*(\Sigma) \cong C^* (\Sigma_1) \oplus C^* (\Sigma_2)$ as $C^*$-algebras. Denoting the canonical unitaries of the two summands by $\delta_i$ for $i = 1,2$, an isomorphism is given by $f \mapsto f_{\upharpoonright A_1} \oplus f_{\upharpoonright A_2}$ for $f \in C(X)$ and $\delta \mapsto \delta_1 \oplus \delta_2$.
\end{lem}
We also make the following definition.
\begin{defn}
Let $p$ be a positive integer.
For a subset $A \subseteq C(\mathbb{T})$ we denote by $\diag(A)$ the subset of $C(\mathbb{T}, M_p)$ consisting of diagonal $p \times p$-matrices with entries from $A$. Given $p$ functions $f_0, f_1, \ldots, f_{p-1} \in C(\mathbb{T})$ we denote by $\diag(f_0, f_1, \ldots, f_{p-1})$ the diagonal $p \times p$-matrix having $f_i$ as entry $(i,i)$.
\end{defn}
\begin{pro}\label{isolbana4}
Suppose that $\Sigma = (X, \sigma)$ is such that $X$ contains an orbit of order $p$, $\mathcal{O}_{\sigma} (x) = \{x, \sigma(x), \ldots, \sigma^{p-1} (x)\}$, consisting of isolated points. Then there exist $C^*$-subalgebras $B_i$ of $C^*(\Sigma)$ with $C(X) \subsetneq B_i \subsetneq C(X)'$ for $i =1,2$ such that $B_1$ has the intersection property for ideals, and $B_2$ does not.
\end{pro}
\begin{proof}
Denote the complement of $\mathcal{O}_{\sigma} (x)$ in $X$ by $X_1$, and let $\sigma_1 = \sigma_{\upharpoonright X_1}$. Using the notation $\Sigma_1 = (X_1, \sigma_1)$ we then have, by Proposition~\ref{enbana4} together with Lemma~\ref{dirsum4}, that $C^* (\Sigma) \cong C^* (\Sigma_1) \oplus C(\mathbb{T},M_p)$, and it is straightforward to check that $C(X)$ is identified with $C(X_1) \oplus \diag(\mathbb{C}) \subseteq C^* (\Sigma_1) \oplus C(\mathbb{T},M_p)$ under this isomorphism. If $X_1$ is empty, the left summands above and in what follows are naturally zero. Since the commutant of $\diag(\mathbb{C})$ in $C(\mathbb{T}, M_p)$ is easily seen to be $\diag(C(\mathbb{T}))$, we conclude that $C(X)'$ is identified with $C(X_1)' \oplus \diag(C(\mathbb{T}))$, where $C(X_1)'$ is the commutant of $C(X_1)$ in $C^* (\Sigma_1)$, under the isomorphism.
Now take two distinct points $x_1, x_2 \in \mathbb{T}$ and consider the $C^*$-subalgebra $B = \{f \in C(\mathbb{T}) : f(x_1) = f(x_2)\}$ of $C(\mathbb{T})$. Let \[B_1 = C(X_1)' \oplus \diag(B) \subseteq C^* (\Sigma_1) \oplus C(\mathbb{T},M_p).\] It is readily checked that $B_1$ is a $C^*$-subalgebra of $C^* (\Sigma_1) \oplus C(\mathbb{T},M_p)$ such that \[C(X_1) \oplus \diag(\mathbb{C}) \subsetneq B_1 \subsetneq C(X_1)' \oplus \diag(C(\mathbb{T})).\] Now let $I \subseteq C^* (\Sigma_1) \oplus C(\mathbb{T},M_p)$ be a non-zero ideal. Then by Corollary~\ref{intprocol4} there exists a non-zero $a \oplus b \in (C(X_1)' \oplus \diag(C(\mathbb{T}))) \cap I$. If $b = 0$ then $a \oplus b \in I \cap B_1$ and we are done, so suppose $b \neq 0$. Writing $ b = \diag(f_0, \ldots, f_{p-1})$ where the $f_i$ are in $C(\mathbb{T})$, we see that at least one of the $f_i$ is such that there exist a $t \in \mathbb{T} \setminus \{x_1, x_2\}$ with $f_i (t) \neq 0$. Take an $f \in C(\mathbb{T})$ such that $f(t) =1$ and $f(x_1) = f(x_2) =0$ and let $M = \diag(f, f, \ldots, f) \in  C(\mathbb{T}, M_p)$. Then $0 \oplus 0 \neq (0 \oplus M) \cdot (a \oplus b) = 0 \oplus M \cdot b\in I \cap B_1$. We conclude that $B_1$ has the intersection property.
To construct an intermediate subalgebra that does not have the intersection property, take two proper closed subsets $C_1, C_2 \subseteq \mathbb{T}$ such that $C_1 \cup C_2 = \mathbb{T}$. Note that connectedness of $\mathbb{T}$ implies that $C_1 \cap C_2 \neq \emptyset$ and that both $C_1$ and $C_2$ are infinite. Now let $B_2 = C(X_1)' 
\oplus \diag(\{f \in C(\mathbb{T}) : f \textup{ is constant on } C_1\})$. It is not hard to see that $B_2$ is a $C^*$-subalgebra of $C^* (\Sigma_1) \oplus C(\mathbb{T},M_p)$ such that \[C(X_1) \oplus \diag(\mathbb{C}) \subsetneq B_2 \subsetneq C(X_1)' \oplus \diag(C(\mathbb{T})).\] Now let $I = \{0\} \oplus M_p(\ker(C_2)) \subseteq C^*(\Sigma_1) \oplus C(\mathbb{T}, M_p)$, where $M_p (\ker(C_2))$ is the set of $p \times p$-matrices with entries in $\ker(C_2)$. This is easily seen to be a non-trivial closed ideal and clearly $I \cap B_2 =\{0 \oplus 0\}$ as $C_1 \cap C_2 \neq \emptyset$.
\end{proof}
We can now easily derive the main result of this section.
\begin{thm}\label{mellankompl4}
For a topological dynamical system $\Sigma = (X, \sigma)$, $C(X)'$ always has the intersection property for ideals in the associated $C^*$-crossed product $C^*(\Sigma)$. 
Furthermore, precisely one of the following two cases occurs:
\begin{enumerate}
\item $\Sigma$ is topologically free. Then $C(X) = C(X)'$;
\item $\Sigma$ is not topologically free. Then $C(X) \subsetneq C(X)'$ and $C(X)$ does not have the intersection property for ideals. In this case, there exist $C^*$-subalgebras $B_i$ with $C(X) \subsetneq B_i \subsetneq C(X)'$ for $i =1,2$ such that $B_1$ has the intersection property for ideals and $B_2$ does not.
\end{enumerate}
\end{thm}
\begin{proof}
That $C(X)'$ always has the intersection property for ideals is stated in Corollary~\ref{intprocol4}.
Case (i) is clear by Theorem~\ref{triquivc4}. If $\Sigma$ is not topologically free, $C(X)$ is properly contained in $C(X)'$ and does not have the intersection property for ideals by Theorem~\ref{triquivc4}. The explicit description of $C(X)'$ in Proposition~\ref{commbesk4}(i)  implies that there exists a positive integer $n$ such that $\Per^n (\sigma)^0 \neq \emptyset$. Suppose first that $\Per^n (\sigma)^0$ contains a non-isolated point. Then it is easy to see that the conditions in Propositions \ref{ejintp4} and \ref{intp4} are both satisfied and case (ii) follows. If $\Per^n (\sigma)^0$ contains only isolated points, the condition in Proposition~\ref{isolbana4} is satisfied and again case (ii) follows.
\end{proof}
\section{Projections onto $C(X)'$}\label{projn4}
With the norm one projection $E : C^* (\Sigma) \rightarrow C(X)$ in mind, one might wonder whether there exists a projection map of norm one onto $C(X)'$. We have the following result.
\begin{thm}
A necessary and sufficient condition for the existence of a
 projection of norm one from $C^* (\Sigma)$ onto $C(X)'$ is that for every positive integer
 $k$ the set $\Per_k(\sigma)^0$ is closed.
In this case the projection is uniquely determined and it is faithful.
\end{thm}
As we feel that the proof of this result is rather long in relation to its relevance, we omit it here and content ourselves with a proof in the special case when $\Sigma = (X, \sigma)$ is such that $\Sigma = \Per_q (\sigma)$ for some positive integer $q$.
\begin{pro}
Suppose $\Sigma = (X, \sigma)$ is such that $\Sigma = \Per_q (\sigma)$ for some positive integer $q$. Then there exists a unique projection $E_0$ of norm one from $ C^* (\Sigma)$ onto $C(X)'$. Furthermore, $E_0$ is faithful.
\end{pro}
\begin{proof}
Here $\PIP(\sigma) = \Per_q (\sigma) = X$.
It follows from Proposition~\ref{commbesk4}(i) that \[C(X)' = \{a \in C^*(\Sigma) : a(j) = 0 \textup{ if } q \textup{ does not divide }j\}.\]
We define a linear map $E_0$ from the algebraic part of $C^* (\Sigma)$ to $C(X)'$ by \[E_0 (\sum_{k=-n}^n f_k \delta^k) = \sum_{\{l : |lq| \leq n\}} f_{lq} \delta^{lq} \textup{ for } f_k \in C(X) \textup{ and } n \textup{ a non-negative integer}.\] Note that the image of $E_0$ is dense in $C(X)'$ by Proposition~\ref{commbesk4}(i) and Proposition~\ref{cesar4}. We will show that $E_0$ is norm-decreasing and then extend it by continuity to the whole of $C^*(\Sigma)$. First we show that for $\varphi \in \Phi$ and any $a$ in the algebraic part of $C^*(\Sigma)$, we have $\varphi (a) = \varphi \circ E_0 (a)$. Write $a = \sum_{k=-n}^n f_k \delta^k$ and take a  $\varphi_{y,t}$ with $y \in \Per_q (\sigma)$ and $t \in \mathbb{T}$. Denote by $\tilde{\pi}_{y,t}$ its  associated irreducible GNS-representation with cyclic unit vector $\xi_{y,t}$ and note that 
\[\varphi_{y,t} (a) = (\tilde{\pi}_{y,t} (a) \xi_{y,t}, \xi_{y,t}) = \sum_{k=-n}^n (\tilde{\pi}_{y,t}(f_k \delta^k) \xi_{y,t}, \xi_{y,t})  
= \sum_{\{l : |lq| \leq n\}} f_{lq} (y) t^l.\]
On the other hand
\[\varphi_{y,t} \circ E_0 (a) = \sum_{\{l : |lq| \leq n\}} (\tilde{\pi}_{y,t}(f_{lq} \delta^{lq}) \xi_{y,t}, \xi_{y,t}) = \sum_{\{l : |lq| \leq n\}} f_{lq} (y) t^l.\]
Now as $E_0 (a) \in C(X)'$, it follows by Theorem~\ref{comcom4}(ii) and the above conclusion that \[\|E_0 (a)\| = \sup_{\varphi \in \Phi} |\varphi \circ E_0 (a)| = \sup_{\varphi \in \Phi} |\varphi (a)| \leq \|a\|.\]
So indeed $E_0$ is norm decreasing on algebraic elements $a$ and thus extends by continuity to a norm one projection from $C^* (\Sigma)$ onto $C(X)'$. To see that $E_0$ is faithful, let $b \in C^* (\Sigma)$ be a non-negative element and suppose $E_0 (b) = 0$. By definition, $\gamma(y,t) \circ E_0 = \varphi_{y,t} \circ E_0$ and by the above $\varphi_{y,t} \circ E_0 = \varphi_{y,t}$, so $\varphi_{y,t}(b) = 0$. Since every element of $\Phi$ has this form and $\Phi$ is total by Lemma~\ref{total4}, it follows that $b = 0$ and hence $E_0$ is faithful. Suppose there is another norm one projection $E_1 : C^* (\Sigma) \rightarrow C(X)'$. Since $X = \Per_q(\sigma)$ it follows easily that the system $\tilde{\Sigma} = (\Gamma, \tilde{\sigma})$ as introduced in the passage following Lemma~\ref{vilkafinns4} above is such that $\Gamma = \Per_q (\tilde{\sigma}) = \PIP(\tilde{\sigma})$ and hence by Proposition~\ref{unutv4} every element of $\Gamma$ has a unique pure state extension to $C^* (\Sigma)$. Letting $\gamma(y,t) \in \Gamma$ be arbitrary, we see that $\gamma(y,t) \circ E_1$ is a state extension of it and since the unique pure state extension of $\gamma(y,t)$ must also be the unique state extension of it, it follows that $\gamma(y,t) \circ E_1 = \gamma(y,t) \circ E_0 = \varphi_{y,t}$. Since $\gamma(y,t)$ was arbitrary it follows that $E_1 = E_0$ as desired.
\end{proof}
\subsection*{Acknowledgements}
The authors are grateful to Marcel de Jeu, Takeshi Katsura and Sergei Silvestrov for fruitful discussions and useful suggestions for improvements of the paper, and to the referee for giving detailed suggestions concerning the exposition of the material. This work was supported by a visitor's grant of the Netherlands Organisation for Scientific Research (NWO) and The Swedish Foundation for International Cooperation in Research and Higher Education (STINT).

\end{document}